\documentclass[reqno]{amsart}
\usepackage{amsmath}
\usepackage{amsfonts}
\usepackage{amsthm}
\usepackage{amssymb}
\usepackage{graphicx}
\usepackage{bm}
\usepackage{dsfont}
\usepackage{color}
\usepackage{enumerate}
\usepackage[font=footnotesize]{caption}
\usepackage[all]{xy}
\usepackage{tikz}
\usepackage{pgfplots}
\allowdisplaybreaks

\numberwithin{equation}{section}

\newtheorem{thm}{Theorem}[section]
\newtheorem*{thm*}{Theorem}
\newtheorem{cor}[thm]{Corollary}
\newtheorem{lem}[thm]{Lemma}

\theoremstyle{definition}
\newtheorem{rem}[thm]{Remark}

\newtheorem{dfn}[thm]{Definition}
\newtheorem{rmk}[thm]{Remark}

\newcommand{\N}{\mathds{N}}

\newcommand{\Z}{\mathds{Z}}

\newcommand{\R}{\mathds{R}}

\newcommand{\T}{\mathds{T}}

\newcommand{\diff}{\mathrm{d}}

\newcommand{\con}{\mathrm{con}}

\newcommand{\ta}{{\mathrm T}}

\begin{document}

\title[Integrable magnetic flows on the two-torus]{Integrable magnetic flows on the two-torus:\\ Zoll examples and systolic inequalities}

\author[L. Asselle]{Luca Asselle}
\address{Luca Asselle\newline\indent
Justus Liebig Universit\"at Gie\ss en, Mathematisches Institut \newline\indent Arndtstra\ss e 2, 35392 Gie\ss en, Germany}
\email{luca.asselle@math.uni-giessen.de}

\author[G. Benedetti]{Gabriele Benedetti}
\address{Gabriele Benedetti\newline\indent 
Universit\"at Heidelberg, Mathematisches Institut \newline\indent  
Im Neuenheimer Feld 205, 69120 Heidelberg, Germany
}
\email{gbenedetti@mathi.uni-heidelberg.de}

\date{\today}

\begin{abstract}
In this paper we study some aspects of integrable magnetic systems on the two-torus. On the one hand, we construct the first non-trivial examples with the property that all magnetic geodesics with unit speed are closed. On the other hand, we show that those integrable magnetic systems admitting a global surface of section satisfy a sharp systolic inequality. 
\end{abstract}

\maketitle


\vspace{-3mm}

\section{Introduction}

A \textit{magnetic system} on a closed oriented surface $\Sigma$ is a pair $(g,b)$, where $g$ is a Riemannian metric and $b:\Sigma\to \R$ is a smooth function, which we refer to as the \textit{magnetic function}. Every magnetic system yields a flow on $\mathrm S\Sigma$, the unit tangent bundle of $\Sigma$ with respect to $g$, by
\[
\Phi_{g,b}^t (q,v) = (\gamma(t),\dot \gamma(t)), \quad \forall t\in \R,
\]
where $\gamma:\R\to \Sigma$ is the unique $(g,b)$-\textit{geodesic}, that is, unit-speed curve satisfying the prescribed geodesic curvature equation
\begin{equation}
\kappa_\gamma(t)=b(\gamma(t)),\label{magneticgeodesics}
\end{equation}
with initial conditions $(\gamma(0),\dot\gamma(0))=(q,v)$.
Here $\kappa_\gamma$ is the geodesic curvature of $\gamma$ with respect to $g$ and the given orientation of the surface. For every $\lambda>0$, the solution of \eqref{magneticgeodesics} for the pairs $(\lambda^2g,\lambda b)$ and $(g,b)$ coincide. Moreover, changing simultaneously the orientation and the sign of $b$ does not change the solutions of \eqref{magneticgeodesics}, therefore we will assume throughout the paper that the average of $b$
\[
\langle b\rangle:=\frac{1}{\mathrm{area}(\Sigma,g)}\int_\Sigma b\mu_g,
\]
is non-negative. Here $\mu_g$ is the area form given by $g$ and the given orientation on $\Sigma$, and $\mathrm{area}(\Sigma,g):=\int_\Sigma\mu_g$.

We readily see that, for every $\lambda>0$, $(g,\lambda b)$-geodesics correspond to $\lambda^{-1}$-speed solutions of
\begin{equation}\label{magneticgeodesics2}
\nabla_{\dot\gamma}\dot\gamma=(b\circ\gamma)\cdot \dot \gamma^\perp,
\end{equation}
where $\nabla$ denotes the Levi-Civita connection, and $\dot \gamma^\perp$ is given by rotating $\dot \gamma$ by $\frac \pi 2$ in $\ta_\gamma\Sigma$.

This enables us to study $(g,b)$-geodesics with symplectic methods. Indeed, solutions of \eqref{magneticgeodesics2} correspond, up to the isomorphism $\ta\Sigma\to\ta^*\Sigma$ given by the metric $g$, to the solutions of the Hamiltonian system on $\ta^*\Sigma$ with kinetic Hamiltonian function
\begin{equation}\label{eq:H}
H:\ta^*\Sigma\to \R,\qquad H(q,p)=\frac 12|p|^2
\end{equation}
and twisted symplectic form
\begin{equation}\label{eq:Omega}
\Omega_{g,b} := \diff p \wedge \diff q - \pi^* (b \mu_g),
\end{equation}
where $\pi:\ta^*\Sigma\to \Sigma$ is the bundle projection. Such Hamiltonian systems are of physical interest since they model the 
motion of a charged particle under the effect of the magnetic field $b\mu_g$.

The easiest examples of magnetic systems are given by pairs $(g_{\mathrm{con}},b_{\mathrm{con}})$, where $g_{\mathrm{con}}$ has constant Gaussian curvature $K_{\mathrm{con}}$ and $b_{\mathrm{con}}$ is constant. We see that, for
\begin{equation}\label{poscurv}
b_{\mathrm{con}}^2+K_{\mathrm{con}}>0,
\end{equation}
 $(g_{\mathrm{con}},b_{\mathrm con})$-geodesics are given by geodesic spheres with radius $r=r(K_{\mathrm{con}},b_{\mathrm{con}})>0$. In particular, the corresponding flow 
$\Phi_{g_{\mathrm{con}},b_{\mathrm{con}}}^t$ yields a free $S^1$-action on $\mathrm S\Sigma$, whose quotient map associates to every geodesic sphere its center. We are prompted, therefore, to make the following general definition. 

\begin{dfn}
A magnetic system $(g,b)$ is called Zoll, if $\Phi_{g,b}^t$ yields a free $S^1$-action on $\mathrm S\Sigma$.
\end{dfn}

\begin{rmk}
If $\Sigma\neq S^2$, the classification of Seifert fibrations yields that  $(g,b)$ is Zoll if and only if all $(g,b)$-geodesics are closed. For $\Sigma=S^2$ there are instead 
examples where all $(g,b)$-geodesics are closed but the induced $S^1$-action on $\mathrm SS^2$ is only semi-free \cite{Benedetti:2016sf}.\qed
\end{rmk}

\begin{rmk}
In the particular case $b=0$ we retrieve the definition of Zoll metric on $\Sigma$, and in this case we necessarily have $\Sigma=S^2$. The space of Zoll metrics on $S^2$ has already 
been thoroughly investigated in the past century, and the local structure around $g_{\mathrm{con}}$ is nowadays well understood: Every smooth one-parameter family $s\mapsto\rho_s:S^2\to\R$ of functions with $\rho_0=0$ yields a family of metrics $s\mapsto g_s:=e^{\rho_s}g_{\mathrm{con}}$. 
Funk showed that if $g_s$ is Zoll, then $\frac{\diff\rho_s}{\diff s}(0)$ must be an odd function on $S^2$ \cite{Fun13}. Conversely, using the Nash-Moser implicit function theorem, Guillemin proved that for every odd function $f:S^2\to\R$ there is a deformation of Zoll metrics as above with $f=\frac{\diff\rho_s}{\diff s}(0)$ \cite{Gui76}. In particular, the space 
of Zoll metrics on $S^2$ is infinite dimensional. We shall notice that already the space of Zoll metrics obtained from surfaces of revolutions in $\R^3$ is infinite dimensional, as the very explicit examples constructed by Zoll in \cite{Zol03} demonstrate.\qed
\end{rmk}

If $(g,b)$ is Zoll, then \eqref{poscurv} generalizes to
\begin{equation*}
\overline K(g,b):=\langle b\rangle ^2+\frac{2\pi\chi(\Sigma)}{\mathrm{area}(\Sigma,g)}>0
\end{equation*}
and $(g,b)$-geodesics belong to the class of closed curves $\Lambda^+(\Sigma)$ on $\Sigma$ whose lift $(\gamma,\dot\gamma)$ on $\mathrm S\Sigma$ is homotopic to an oriented fiber of $\pi:\mathrm S\Sigma\to\Sigma$. Moreover, as shown in \cite{Benedetti:2018c}, 
Zoll magnetic systems are the protagonist of a sharp local magnetic systolic inequality as we now recall.

For any magnetic system, $(g,b)$-geodesics in the class $\Lambda^+(\Sigma)$ are critical points of the functional
\begin{equation}
\ell_{(g,b)}:\Lambda^+(\Sigma)\to\R,\qquad \ell_{(g,b)}(\gamma):=\ell_g(\gamma)-\int_{D^2}\Gamma^*(b\mu_g),
\end{equation}
where $\ell_g$ is the $g$-length and $\Gamma:D^2\to\Sigma$ is a capping disc for $\gamma$ obtained projecting a homotopy between $(\gamma,\dot\gamma)$ and a $\pi$-fiber on $\mathrm S\Sigma$. We denote by $\Lambda^+(g,b)$ the (possibly empty) set of $(g,b)$-geodesics in $\Lambda^+(\Sigma)$. Theorem 1.9 and Remark 1.12 in \cite{Benedetti:2018c} show that if $(g,b)$ is $C^3$-close to some $(g_0,b_0)$ Zoll, then
\[
\inf_{\gamma\in\Lambda^+(g,b)}\ell_{(g,b)}(\gamma)\leq \frac{2\pi}{\langle b\rangle+\sqrt{\strut\overline K(g,b)}}
\]
with equality if and only if $(g,b)$ is Zoll. As a special case, one recovers in a stronger topology the local systolic inequality for Zoll metrics on $S^2$ proved in \cite{ABHS17} (see also \cite{ABHS18b}).

The local magnetic systolic inequality leaves the door open to a number of interesting questions. First, for applications of the local inequality it urges us to know more about the space of Zoll magnetic systems. 
As Zoll did in the purely Riemannian case, one could start looking for non-trivial Zoll magnetic systems on $S^2$ or $\T^2$ which are rotationally symmetric, namely of the form
\begin{equation}\label{eq:sym}
g=\diff x^2+a(x)^2\diff y^2,\qquad b=b(x),
\end{equation}
where $a$ and $b$ are functions of the $x$-variable only.

Second, one would like to understand if the systolic inequality holds true also for magnetic systems which are not close to a Zoll one. Already in the purely Riemannian case this is known to be false, as the Calabi-Croke sphere demonstrates; see e.g. \cite{Sabourau:2010}. Nevertheless, Abbondandolo, Bramham, Hryniewicz and Salom\~ao did establish the global inequality in the space of rotationally symmetric metrics on $S^2$ \cite{ABHS18}. Thus, also for this question, we are prompted to look at rotationally symmetric magnetic systems first.

\section{Statement of results}
We can now present our contribution to the two questions above.
\subsection{Zoll magnetic systems on flat tori}
Our first result constructs non-trivial 1-parameter families of rotationally symmetric Zoll systems on certain flat tori. Throughout the paper, we will denote with $\T^2$ the torus obtained by quotienting $\R^2$ by the lattice $\Z\times L\Z\subset\R^2$ for some $L>0$, and with $g_{\mathrm{con}}$ the corresponding flat metric. Also, we define the countable dense set of positive numbers
\begin{equation}
\mathcal B:=\{2\pi n/\xi\ |\ n\in \N,\ \xi\in(0,\infty),\ J_1(\xi)=0\},
\label{setB}
\end{equation}
where $J_1$ is the first Bessel function of the first kind. 

\begin{thm}\label{thm:main}
Assume that $b_{\mathrm{con}}=2\pi n/\xi\in \mathcal B$. Then there exists a family $s\mapsto b_s$ of magnetic functions $b_s:\T^2\to (0,\infty)$, $s\in(-b_{\mathrm{con}}^{-1},b_{\mathrm{con}}^{-1})$, depending only on the $x$-variable such that $b_0= b_{\mathrm{con}}$, $\langle b_s\rangle=b_{\mathrm{con}}$, and $(g_{\mathrm{con}},b_s)$ is Zoll for all $s\in(-b_{\mathrm{con}}^{-1},b_{\mathrm{con}}^{-1})$. The deformation is non-trivial and unbounded:
\[
\frac{\diff b_s}{\diff s}(0)\neq0,\qquad \lim_{s\mapsto \pm b_{\mathrm{con}}^{-1}}\max b_s=\infty.
\]
If $b_{\mathrm{con}}\notin \mathcal B$, then the only magnetic function $b$ depending only on $x$ with $\langle b\rangle=b_{\mathrm{con}}$ 
such that $(g_\con,b)$ is Zoll is given by $b=b_{\mathrm{con}}$. 
\end{thm}

The examples given by Theorem \ref{thm:main} are very explicit since the symmetry yields a multivalued first integral $I:\mathrm S\T^2\to\R/b_{\mathrm{con}}\Z$ that can be used to solve the equations of motion by quadratures. The exact expressions are given in Remark \ref{r:final}. In Figures \ref{figure1} and \ref{figure2} we draw the $(g_{\mathrm{con}},b_s)$-geodesics for $b_{\mathrm{con}}=2\pi/\xi_1$ and $b_{\mathrm{con}}=2\pi/\xi_2$ respectively, where $\xi_1$ and $\xi_2$ are the first two positive zeros of $J_1$. In fact, as can be seen from the formulas for $b_s$ obtained in \eqref{e:Fs} and \eqref{e:bs}, for every $m\in\N$ the systems for $b_{\mathrm{con}}$ and $b'_{\mathrm{con}}=m b_{\mathrm{con}}$ are related by the covering map $(x,y)\mapsto(mx,my)$ of $\T^2$ and so the corresponding geodesics have the same shape. Each plot in the two figures below corresponds to $s=\tfrac{k}{5}b_{\mathrm{con}}^{-1}$ for $k=0,\ldots,5$. In each plot we draw five $(g_{\mathrm{con}},b_s)$-geodesics with different colors representing different values of the first integral $I$. For $k=5$, we have $b_s(0)=+\infty$, which causes the displayed curves to have a sharp edge there. 

\begin{figure}[h]
	\centering
	\includegraphics[width=0.98\textwidth]{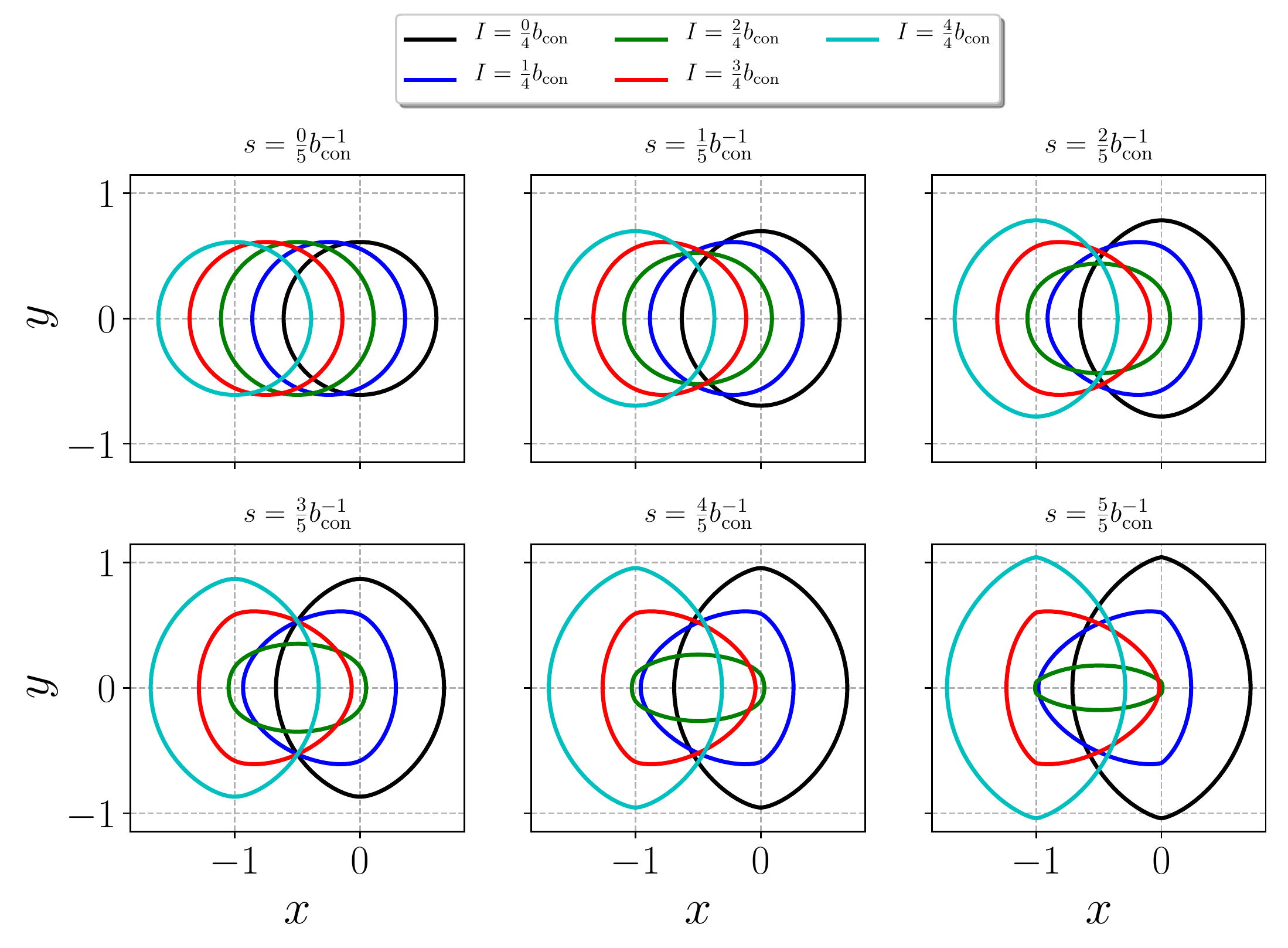}
	\caption{$(g_{\mathrm{con}},b_s)$-geodesics for the $1$-parameter family corresponding to $\xi_1$}
	\label{figure1}
\end{figure}

\begin{figure}[ht]
	\centering
	\includegraphics[width=0.98\textwidth]{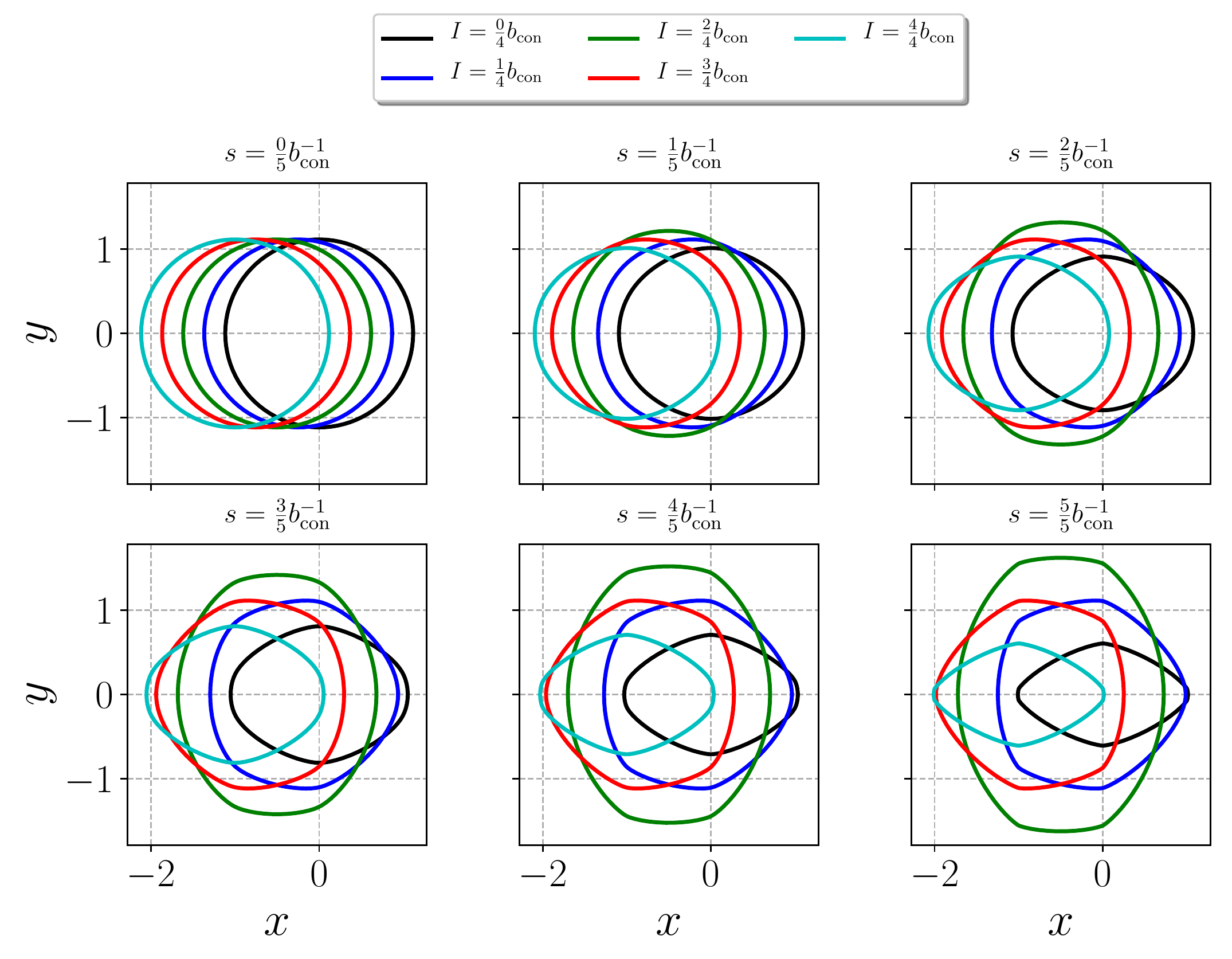}
	\caption{$(g_{\mathrm{con}},b_s)$-geodesics for the $1$-parameter family corresponding to $\xi_2$}
	\label{figure2}
\end{figure}

Let us now briefly comment on the proof of Theorem \ref{thm:main}. The argument hinges on the fact that the flow of Zoll magnetic systems with symmetry has a global surfaces of section. The Zoll condition can then be reformulated by asking that the return map is the identity or, equivalently, that all the Fourier coefficients of the map vanish. Using in an essential fashion that the metric is flat, we are able to compute such Fourier coefficients and show that there exist non-constant functions $b$ with $\langle b\rangle=b_{\mathrm{con}}$ and a return map with zero Fourier coefficients if and only if $b_{\mathrm{con}}\in \mathcal B$. In such a case, the functions $b$ with this property are parametrised by an open set of $\R^{d(b_{\mathrm{con}})}$, where $d(b_{\mathrm{con}})$ is the number of zeros $\xi$ of $J_1$ such that $b_{\mathrm{con}}=2\pi n/\xi$. In particular, if $d(b_{\mathrm{con}})=1$, then the examples of Theorem \ref{thm:main} are the only ones with flat metric and $\langle b\rangle=b_{\mathrm{con}}$. We remark that $d(b_{\mathrm{con}})=1$ for all $b_{\mathrm{con}}\in \mathcal B$ if and only if the ratio of any two positive distinct zeros of $J_1$ were irrational. Unfortunately, this property does not seem to be known.

\vspace{2mm}

Theorem \ref{thm:main} is only the first step in the investigation of Zoll magnetic systems. We do not know, for example, if \text{explicit} rotationally symmetric Zoll magnetic systems can be found on $S^2$ or on $\T^2$ for metrics $g$ different from the flat one considered above. Leaving symmetry behind, one could try to generalize Guillemin's approach to yield an infinite dimensional space of non-trivial Zoll deformations for $(g_{\mathrm{con}},b_{\mathrm{con}})$ on any surface. Such systems, however, will not be explicit as will be obtained from a Nash-Moser implicit function theorem. Having such a result will help us understand the local structure of Zoll magnetic systems at $(g_{\mathrm{con}},b_{\mathrm{con}})$. However, getting an intuition for what the global structure should be seems an even harder task. Is it true, for instance, that for every $g$ there exists a magnetic function $b$ making $(g,b)$ into a Zoll magnetic system?

We conclude this subsection mentioning three results connected with Theorem \ref{thm:main}. The first result is due to Tabachnikov, who constructs $1$-homogeneous Lagrangians $L:\ta\R^2\to\R$, whose Euler-Lagrange solutions are Euclidean circles of given radius \cite{Tab04}. Since the lift to $\R^2$ of the flow of a magnetic system on $\T^2$ is the Euler-Lagrange flow of a particular $1$-homogeneous Lagrangian, we see that Tabachnikov's result is complementary to ours. In his case, the curves are fixed to be Euclidean circles and one let the Lagrangian vary in a large class. In our case, we allow curves of different shapes (as soon as they are all closed) but we look only at Lagrangians coming from a magnetic system. We remark that the zeros of $J_1$ play an important role also in Tabachnikov's work, which suggests that the two constructions might be related. 

The second result is due to Burns and Matveev and asserts that if we require that $(g_0,b_0)$ and $(g_1,b_1)$ have the same solutions of \eqref{magneticgeodesics} and one among $b_0$ and $b_1$ is non-zero, then $(g_1,b_1)=(\lambda^2g_0,\lambda^{-1}b_0)$ for some $\lambda>0$ \cite{BM06}. This means that the flexibility of Theorem \ref{thm:main} and in Tabachnikov's result turns into rigidity, if we both want the curves to be fixed and the Lagrangian to be of magnetic type.

For the third result, we fix a subset $C\subset(0,\infty)$ and look for metrics $g$ on $\Sigma$ and functions $b:\Sigma\to\R$ such that $(g,\lambda b)$ is Zoll for all $\lambda\in C$. This condition means that the magnetic Hamiltonian system with Hamiltonian \eqref{eq:H} and symplectic form \eqref{eq:Omega} is Zoll on every energy level in the set $\tfrac{1}{2C^2}$. Lange and the first author showed in \cite{Asselle:2019b} that, for $\Sigma=\T^2$ and $C=(0,\infty)$, this condition implies that $(g,b)=(g_{\mathrm{con}},b_{\mathrm{con}})$, thus showing that trivial Zoll magnetic systems on $\T^2$ exhibit some rigidity properties. This provides an interesting contrast with the flexibility proved in Theorem \ref{thm:main}.

\subsection{A systolic inequality for symmetric systems on the torus}
Our second result establishes the sharp systolic inequality for rotationally symmetric magnetic systems on $\T^2$ possessing a global torus-like surface of section. The precise statement is as follows.
\begin{thm}\label{thm:main2}
Let $\T^2$ be a two-torus with coordinates $(x,y)$, and consider a symmetric magnetic system $(g,b)$ as in \eqref{eq:sym}, where $a,b:\R/\Z\to\R$ are functions of the $x$-variable only. If the condition
\begin{equation}\label{eq:condition}
\Big|\frac{\diff a}{\diff x}\Big|<ab
\end{equation}
is met, then the systolic inequality
\begin{equation}\label{eq:sys}
\inf_{\gamma\in\Lambda^+(g,b)}\ell_{(g,b)}(\gamma)\leq \frac{\pi}{\langle b\rangle}
\end{equation}
holds. The equality case occurs if and only if $(g,b)$ is Zoll.
\end{thm}
The importance of condition \eqref{eq:condition} will be thoroughly discussed in the next section. Here we mention only two properties which are equivalent to it. The first is that no curve with constant $x$-coordinate is a $(g,b)$-magnetic geodesic. In particular, all symmetric Zoll magnetic systems must satisfy \eqref{eq:condition} since in this case all closed $(g,b)$-geodesics are contractible. The second is that for each $\theta_0\in\R/2\pi\Z$ the embedding
\begin{equation}\label{e:t2}
\iota_{\theta_0}:\T^2\to\mathrm S\T^2,\quad (x,y)\mapsto \cos\theta_0 \partial_x+\frac{\sin\theta_0}{a(x)}\partial_y
\end{equation}
is a torus-like global surface of section for the flow $\Phi_{(g,b)}$ inside $\mathrm S\T^2$, and the associated first-return map can be written quite explicitly in terms of the first integral $I$, a fact that will be key in the proof of Theorem \ref{thm:main2}.

A direction of future research will be to decide if \eqref{eq:sys} holds for symmetric systems not satisfying \eqref{eq:condition}. In this more general case one can still construct annulus-like surfaces of section which, however, are not global. For this reason, already showing that $\Lambda^+(g,b)\neq\varnothing$ is an interesting problem.
\subsection{Structure of the paper}
Section \ref{s:2} is devoted to establishing basic properties of rotationally invariant magnetic systems on $\T^2$ and ends with the proof of Theorem \ref{thm:main2}. In Section \ref{s:3} we prove Theorem \ref{thm:main}.

\section{The first integral $I$}
\label{s:2}

As above we denote with $(x,y)$ angular coordinates on $\T^2=[0,1]\times [0,L]/\sim$, and consider a magnetic system $(g,b)$ as in \eqref{eq:sym}. We write $a:\R/\Z\to(0,\infty)$ and $b:\R/\Z\to \R$ for the warping and magnetic functions respectively, which depend on the $x$-variable only. Up to rescaling the $y$-variable $(x,y)\mapsto(x,cy)$ which changes $L$ to $cL$ and $a$ to $a/c$, we assume that the normalization
\begin{equation}\label{e:norm}
\int_0^1a(x)\diff x=1
\end{equation}
holds. This implies that
\begin{equation}\label{eq:intb}
\langle b\rangle=\int_0^1a(x)b(x)\diff x.
\end{equation}
We define the function $B:\R\to\R$ by the properties
\[
B(0)=0,\qquad B'(x)=a(x)b(x),\quad\forall\,x\in\R
\]
so that
\[
B(x+1)=B(x)+\langle b\rangle,\qquad \forall\,x\in\R.
\]
In view of this, we will also write $B:\R/\Z\to\R/\langle b\rangle\Z$ for the corresponding quotient map. Here and below we use a prime to denote derivatives with respect to the $x$-variable.

From the formula for $g$, the embeddings $\iota_{\theta_0}$ in \eqref{e:t2} are well-defined and we have an angular function $\theta:\mathrm S\T^2\to\R/2\pi\Z$ yielding the oriented angle between an element $v\in\mathrm S\T^2$ and $\partial_x$. If $v=v_x\partial_x+v_y\partial_y$, then
\begin{equation}\label{eq:vxvy}
v_x=\cos\theta,\qquad v_y=\frac{\sin\theta}{a(x)}.
\end{equation}
On the other hand, the flow of $\partial_y$ preserves both the metric $g$ and the magnetic function $b$. Therefore, applying Noether's Theorem, we get a multivalued conserved quantity along $(g,b)$-geodesics.
\begin{lem}\label{lem:int}
The function
\[
I:\mathrm S\T^2\to\R/\langle b\rangle\Z,\qquad I(x,y,\theta)=a(x)\sin\theta-B(x).
\]
is constant along $(g,b)$-geodesics. The critical set of $I$ can be expressed in the coordinates $(x,y,\theta)$ by
\begin{equation}\label{eq:critI}
\mathrm{Crit}\, I=\mathrm{Crit}\, I_-\times(\R/L\Z)\times\{-\pi/2\}\ \sqcup\  \mathrm{Crit}\, I_+\times(\R/L\Z)\times\{+\pi/2\},
\end{equation}
where
\[
I_\pm:\R/\Z\to \R/\langle b\rangle\Z,\qquad I_\pm(x)=\pm a(x)-B(x).
\]
Finally, a $(g,b)$-geodesic parametrizes the circle $\{x=x_0\}$ in the positive (respectively negative) $y$-direction if and only if $x_0$ belongs to $\mathrm{Crit}\, I_+$ (respectively to $\mathrm{Crit}\, I_-$).
\end{lem}
\begin{proof}
The lifted flow of $\partial_y$ to the twisted cotangent bundle $(\ta^*\T^2,\Omega_{(g,b)})$, where $\Omega_{(g,b)}$ is given by \eqref{eq:Omega}, is generated by the multivalued Hamiltonian $p_y-B(x)$. Indeed, 
\begin{align*}
\iota_{\partial_y}\Omega_{(g,b)} &=\iota_{\partial_y}\big(\diff p_x\wedge\diff x+\diff p_y\wedge\diff y\big)-\iota_{\partial_y}\big(a(x)b(x)\diff x\wedge\diff y\big)\\
						&=\diff p_y-a(x)b(x)\diff x\\
						&=\diff\big(p_y-B(x)\big).
\end{align*}
Since the kinetic Hamiltonian $H$ from \eqref{eq:H} is invariant under the flow of $\partial_y$, Noether's theorem implies that $p_y-B(x)$ is invariant under the Hamiltonian flow of $H$. As $\Phi_{(g,b)}$ is the restriction of the Hamiltonian flow of $H$ to $\mathrm S\T^2$, after applying the metric isomorphism $\ta\T^2\to\ta^*\T^2$, and $p_y=a(x)\sin\theta$ holds on $\mathrm S\T^2$, we get that $I$ is a first integral for $\Phi_{(g,b)}$.

We compute
\begin{equation}\label{eq:formI}
\frac{\partial I}{\partial x}=a'(x)\sin\theta-a(x)b(x),\qquad \frac{\partial I}{\partial\theta}=a(x)\cos\theta.
\end{equation}
The second quantity is zero if and only if $\cos\theta=0$. In this case, $\theta=\pm\pi/2$ and the first quantity is zero if and only if $x$ is a critical point of $I_\pm$.

Finally, a $(g,b)$-geodesic parametrizes $\{x=x_0\}$ if and only if the corresponding trajectory of the Hamiltonian flow of $H$ is a reparametrization of a trajectory of the Hamiltonian flow of $p_y-B(x)$. By the definition of Hamiltonian vector fields, this happens if and only if $\diff(p_y-B(x))$ and $\diff H$ have the same kernel at this orbit. Since $\mathrm S\T^2$ corresponds to $H^{-1}(1/2)$ and the restriction of $p_y-B(x)$ to this level set corresponds to $I$, we see that the above conditions are equivalent to the fact that the orbit is contained in the critical set of $I$. By \eqref{eq:critI} the assertion follows.   
\end{proof}

\begin{cor}
The equations of motion of a unit-speed $(g,b)$-geodesic $(x,y):\R\to\T^2$, whose tangent vector makes an angle $\theta:\R\to\R/2\pi\Z$ with $\partial_x$, read
\begin{equation}
\left \{\begin{aligned}\dot x &= \cos \theta, \\ \dot y &= \displaystyle \frac{\sin \theta}{a(x)}, \\ \dot \theta &= b(x) - \frac{a'(x)}{a(x)} \sin \theta.
\end{aligned}\right.
\label{hamiltoneq}
\end{equation}
\end{cor}

\begin{proof}
The first two equations follows immediately from \eqref{eq:vxvy}. The third equation is equivalent to
\begin{equation}\label{eq:dottheta}
\dot\theta=-\frac{1}{a}\frac{\partial I}{\partial x}
\end{equation}
Let us assume by contradiction that this equation does not hold. Since $I$ is a first integral, we compute using the first equation in \eqref{hamiltoneq}
\begin{equation}\label{eq:dotI}
0=\dot I=\frac{\partial I}{\partial x}\dot x+\frac{\partial I}{\partial \theta}\dot \theta=\cos\theta\Big(a\dot\theta+\frac{\partial I}{\partial x}\Big).
\end{equation}
Then $\cos\theta\equiv 0$ and $x\equiv x_0\in\R/\Z$, $\theta\equiv\pm\tfrac{\pi}{2}$. According to Lemma \ref{lem:int}, this means that $I_\pm$ has a critical point at $x_0$. However, using the assumption that we want to contradict and the fact that $\dot\theta=0$, we get
\[
\frac{\diff I_\pm}{\diff x}(x_0)=\frac{\partial I}{\partial x}(x_0,\pm\pi/2)\neq0.\qedhere
\]
\end{proof}
As observed in Section 2, Lemma \ref{lem:int} implies that if $(g,b)$ is Zoll, then $\mathrm{Crit}I=\varnothing$. This prompts us to better investigate this condition.
\begin{lem}\label{l:critI}
The following conditions are equivalent:
\begin{enumerate}
\item $\mathrm{Crit}I=\varnothing$;\vspace{3pt}
\item $|a'|<ab$;\vspace{3pt}
\item $\frac{\partial I}{\partial x}<0$;\vspace{3pt}
\item the function $\theta$ is strictly monotone increasing along all $(g,b)$-geodesics;\vspace{3pt}
\item $-\diff\theta\wedge \omega_{(g,b)}$ is a volume form on $\mathrm S\T^2$, where $\omega_{(g,b)}$ is the pull-back of $\Omega_{(g,b)}$ on $\mathrm S\T^2$;\vspace{3pt}
\item for each $\theta_0\in\R/2\pi\Z$ the torus embedding $\iota_{\theta_0}$ from \eqref{e:t2} is a global surface of section for $\Phi_{(g,b)}$.
\end{enumerate} 
\end{lem}

\begin{proof}
$(1)\Leftrightarrow(2)$: The critical set of $I$ is empty if and only if the critical sets of $I_-$ and $I_+$ are empty. This means that $\pm a'-ab\neq0$. By continuity, we either have $\pm a'-ab>0$ or $\pm a'-ab<0$. However, 
\[
\int_0^1(\pm a'-ab)\diff x=-\int_0^1ab\,\diff x=-\langle b\rangle\leq 0,
\]
which implies that only the second possibility can hold.

$(2)\Leftrightarrow(3)$: $|a'|<ab$ if and only if $a'\sin\theta-ab<0$ for all $\theta$. By \eqref{eq:formI} the conclusion follows.

$(3)\Leftrightarrow(4)$: It follows at once from \eqref{eq:dottheta}.

$(3)\Leftrightarrow(5)$: We compute
\begin{align*}
\omega_{(g,b)} &=\diff(\cos\theta\diff x+a\sin\theta\diff y)-ab\diff x\wedge\diff y\\
			&=(a'\sin\theta-ab)\diff x\wedge\diff y+\diff\theta\wedge(-\sin\theta\diff x+a\cos\theta\diff y).
\end{align*}
Therefore, $-\diff\theta\wedge \omega_{(g,b)}=\tfrac{\partial I}{\partial x}\diff\theta\wedge \diff x\wedge\diff y$ from which we see the desired equivalence.

$(4)\Rightarrow(6)$: If $\theta$ is strictly monotone increasing along all flow lines of $\Phi_{(g,b)}$, then for every $\theta_0\in\R/2\pi\Z$ the torus $\R/\Z\times(\R/L\Z)\times\{\theta_0\}$ is a global surface of section.

$(6)\Rightarrow(1)$: Assume that $\iota_{\theta_0}$ is a global surface of section. Then,
\begin{enumerate}[(i)]
\item no orbit of $\Phi_{(g,b)}$ can be entirely contained in $\iota_{\theta_0}$;
\item all orbits of $\Phi_{(g,b)}$ must intersect $\iota_{\theta_0}$.
\end{enumerate}
If $\theta_0\neq\pm\pi/2$, then (ii) is sufficient to see that $\mathrm{Crit}I_\pm=\varnothing$. If $\theta_0=\pi/2$ (the case $\theta_0=-\pi/2$ being analogous), then from (i) we see that $\mathrm{Crit}I_+=\varnothing$ and from (ii) that $\mathrm{Crit}I_-=\varnothing$. 
\end{proof}
\begin{rem}
Condition (e) implies that the odd-symplectic form $\omega_{(g,b)}$ is stable. Even if we will not make use of this in the present paper, let us observe that this has the consequence that an arbitrary magnetic system $(g_1,b_1)$ close to $(g,b)$ satisfying (a) will also have an associated stable odd-symplectic form and therefore it will possess a contractible $(g_1,b_1)$-geodesic according to \cite{Asselle:2014hc}. \qed
\end{rem}

\begin{rem}
A sufficient but not necessary condition for $\mathrm{Crit}I=\varnothing$ is the positivity of the magnetic curvature of $(g,b)$. This phenomenon is analogous to the one ensuring that a convex Riemannian sphere of revolution admits exactly one equator (see also \cite[Proposition 4.6]{Ben16} for the case of magnetic spheres). Let us recall that the magnetic curvature, which in general can be obtained from Jacobi fields along $(g,b)$-geodesics, in the symmetric setting takes the form
\[
K_{(g,b)}:\mathrm S\T^2\to\R,\qquad K_{(g,b)}(x,y,\theta) = b^2(x) + b'(x) \sin \varphi - \frac{a''(x)}{a(x)}.
\]
If we assume that $K_{(g,b)}$ is everywhere positive, then we claim that all critical points of $I_-$ and $I_+$ must be non-degenerate maxima. This implies that the set of critical points of $I_-$ and $I_+$ is empty. To prove the claim, we take a critical point $x_0$ for $I_+$ (the other case being completely analogous) such that $a'(x_0)-a(x_0)b(x_0)=0$. Using this equation, we see that the second derivative is negative
\begin{align*}
	\frac{\diff^2}{\diff x^2}\Big |_{x=x_0} I_+	& = a''(x_0)-a'(x_0)b(x_0) - a(x_0)b'(x_0)\\
	&=- a(x_0) \Big ( f^2(x_0) + f'(x_0) - \frac{a''(x_0)}{a(x_0)} \Big )\\
	\hspace{58mm} &= - a(x_0) \cdot K_{g,b}(x_0,y,\pi/2). \hspace{59mm} \qed
	\end{align*}
\end{rem}

From Lemma \ref{l:critI}, we see that $\mathrm{Crit}I=\varnothing$ yields some strong dynamical consequences. First of all, thanks to Condition (c), we have a function $x:\R/\langle b\rangle\Z\times\R/2\pi\Z\to\R/\Z$ which is the unique solution of $I=I(x(I,\theta),\theta)$ for every $(I,\theta)$. Differentiating this relation with respect to $\theta$ we get 
\begin{equation}\label{eq:xtheta}
\frac{\partial x}{\partial \theta}=-a(x)\cos\theta\frac{\partial x}{\partial I}.
\end{equation}
Using the function $x$ we can parametrize $\mathrm S\T^2$ by the coordinates $(I,y,\theta)$. Moreover, by Condition (d), we can parametrize every $(g,b)$-geodesic by the angle $\theta$. If we do so, then the only non-trivial differential equation in the $(I,y,\theta)$ coordinates is the one determining $y$. Thanks to \eqref{hamiltoneq} and \eqref{eq:dottheta}, it reads
\begin{equation}\label{e:ytheta}
\frac{\diff y}{\diff \theta}=-\sin\theta\frac{\partial x}{\partial I}.
\end{equation}
If we parametrize the global surfaces of section $\theta^{-1}(\theta_0)$ appearing in Condition (f) using coordinates $(I,y)$, then we have a first return map of the form $(I,y)\mapsto(I,y+\Delta(I))$, where $\Delta:\R/\langle b\rangle\Z\to\R$ is the function
\[
\Delta(I):=\int_{\frac{\R}{2\pi\Z}}\frac{\diff y}{\diff \theta}\diff\theta=-\int_{\frac{\R}{2\pi\Z}}\sin\theta\frac{\partial x}{\partial I}(I,\theta)\diff\theta.
\]
\begin{lem}\label{l:SI}
Assume that $\mathrm{Crit}I=\varnothing$ and let $\Delta$ be defined as above. Then
\[
\frac{\diff S}{\diff I}=\Delta,
\]
where $S:\R/\langle b\rangle\Z\to\R$ is the function
\begin{equation}\label{e:S}
S(I):=-\int_{\frac{\R}{2\pi\Z}}a(x(I,\theta))\cos^2\theta\frac{\partial x}{\partial I}(I,\theta)\diff\theta.
\end{equation}
\end{lem}
\begin{proof}
We define $\tilde S:\R\to\R$ by
\[
\tilde S(I)=-\int_{\frac{\R}{2\pi\Z}}x(I,\theta)\sin\theta\diff\theta,
\]
where we identify $x$ with its lift to $\R$. It is clear that $\tilde S$ is a primitive of the lift of $\Delta$. We are left to prove that $\tilde S$ is the lift of $S$. Integrating by parts, we get 
\[
\tilde S(I)=\int_{\frac{\R}{2\pi\Z}}\frac{\partial x}{\partial \theta}\cos\theta\diff\theta=\int_{\frac{\R}{2\pi\Z}}-a(x)\cos\theta\frac{\partial x}{\partial I}\cos\theta\diff\theta=S(I),
\]
where we used \eqref{eq:xtheta} in the second equality.
\end{proof}

We collect the main properties of $S$ in the next lemma.

\begin{lem}
Assume that $\mathrm{Crit}I=\varnothing$ and let $S$ be defined as above. Then
\[
\int_{\frac{\R}{\langle b\rangle\Z}}S(I)\diff I=\pi.
\]
Moreover, an element $I\in\R/\langle b\rangle\Z$ is a critical point of $S$ if and only if the $(g,b)$-geodesics $\gamma_I$ with integral $I$ belong to $\Lambda^+(\T^2)$ (namely are contractible and with turning number $+1$). Moreover,
\[
\ell_{(g,b)}(\gamma_I)=S(I).
\]
In particular, $(g,b)$ is Zoll if and only if $S$ is constant (and equal to $\pi/\langle b\rangle$).
\end{lem}

\begin{proof}
We compute
\begin{align*}
\int_{\frac{\R}{\langle b\rangle\Z}}S(I)\diff I &=\int_{\frac{\R}{2\pi\Z}}\cos^2\theta\Big(\int_{\frac{\R}{\langle b\rangle\Z}}-a(x(I,\theta))\frac{\partial x}{\partial I}\diff I\Big)\diff\theta\\
				&=\int_{\frac{\R}{2\pi\Z}}\cos^2\theta\Big(\int_{\frac{\R}{\Z}}a(x)\diff x\Big)\diff\theta\\
				&=\int_{\frac{\R}{2\pi\Z}}\cos^2\theta\diff\theta\\
				&=\pi,
\end{align*}
where we used \eqref{e:norm} and the fact that for every $\theta$ the map $I\mapsto x(I,\theta)$ reverses the orientation.

From the definition of the first-return map we see that $\Delta(I)=0$ if and only if the $(g,b)$-geodesics with integral $I$ are contractible (and with turning number $+1$). By Lemma \ref{l:SI}, the correspondence between critical points of $S$ and elements in $\Lambda^+(g,b)$ follows. It remains to compute
\[
\ell_{(g,b)}(\gamma_I)=\ell_g(\gamma_I)-\int_{D^2}\Gamma^*_I(ab\diff x\wedge\diff y).
\]
For the first term we use \eqref{eq:dottheta}
\[
\ell_g(\gamma_I)=\int_{\frac{\R}{2\pi\Z}}\frac{1}{\dot\theta}\diff\theta=-\int_{\frac{\R}{2\pi\Z}}a(x)\frac{\partial x}{\partial I}\diff\theta.
\]
For the second term we use Stokes Theorem, the definition of the first integral and \eqref{e:ytheta}
\begin{align*}
\int_{D^2}\Gamma^*_I(ab\diff x\wedge\diff y)& =\int_{\frac{\R}{2\pi\Z}}\gamma_I^*(B(x)\diff y)\\
						&=\int_{\frac{\R}{2\pi\Z}}(a(x)\sin\theta-I)\diff y\\
						&=\int_{\frac{\R}{2\pi\Z}}a(x)\sin\theta\frac{\diff y}{\diff\theta}\diff\theta-I\int_{\frac{\R}{2\pi\Z}}\diff y\\
						&=-\int_{\frac{\R}{2\pi\Z}}a(x)\sin\theta\sin\theta\frac{\partial x}{\partial I}\diff\theta-I\cdot 0.
\end{align*}
Putting the two pieces together yields the desired formula.

Finally, the assertion about Zoll systems is a consequence of the fact that $\int_{\frac{\R}{\langle b\rangle\Z}}\diff I=\langle b\rangle$.
\end{proof}

\begin{proof}[Proof of Theorem \ref{thm:main2}]
By the previous lemma, we have
\begin{align*}
\inf_{\gamma\in\Lambda^+(g,b)}\ell_{(g,b)}(\gamma) &=\min S \leq \frac{\int_0^{\langle b\rangle}S(I)\diff I}{\langle b\rangle} =\frac{\pi}{\langle b\rangle}.
\end{align*}
The inequality is strict if and only if $S$ is not constant, which happens if and only if $(g,b)$ is not Zoll.
\end{proof}

\section{Construction of symmetric Zoll magnetic systems on the flat torus}
\label{s:3}

From the previous section, we know that a rotationally symmetric $(g,b)$ on the two-torus is Zoll if and only if the function $S:\R/\langle b\rangle\Z\to\R$ is constant (and equal to $\pi/\langle b\rangle$). This condition is equivalent to asking that the Fourier coefficients
\[
\hat S(m)=\int_0^{\langle b\rangle}S(I)e^{-\frac{2\pi im}{\langle b\rangle}I}\diff I
\]
vanish for all $m\in\Z\setminus\{0\}$. These coefficients, although in general not computable, admit an explicit formula when the torus is flat, 
namely when $a\equiv 1$. This formula involves the Bessel function (see \cite{Watson:1944})
\[
J_1:(0,\infty)\to \R,\quad J_1(\xi)=\xi\int_0^{2\pi}\cos\big(\xi\sin\theta\big)\cos^2\theta\diff\theta.
\]
The zeros of $J_1$ are all simple and form a strictly increasing, divergent sequence $(\xi_k)_{k\in\N}$.

\begin{lem}\label{l:fourier}
Let $(g_\con,b)$ be a rotationally symmetric magnetic system with $\mathrm{Crit}I=\varnothing$. Then the function 
$B:\R/\Z\to\R/\langle b\rangle\Z$ is invertible and we have
\[
S(I)=\int_0^{2\pi}f(\sin\theta-I)\cos^2\theta\diff\theta,\qquad f:=\frac{\diff B^{-1}}{\diff u}.
\]
Therefore,
\[
\hat S(m)=\frac{\langle b\rangle}{2\pi |m|}J_1\Big(\frac{2\pi |m|}{\langle b\rangle}\Big)\hat f(-m),\qquad\forall\,m\neq0.
\]
\end{lem}
\begin{proof}
By Lemma \ref{l:critI}, the condition $\mathrm{Crit}I=\varnothing$ is equivalent to the positivity of $b$. Thus, $B$ is invertible and from the definition of $I$, we get $x(I,\theta)=B^{-1}(\sin\theta-I)$. Hence, \eqref{e:S} yields the desired formula for $S$. For the Fourier coefficients we compute:
\begin{align*}
\hat S(m)&=\int_0^{2\pi}\Big(\int_0^{\langle b\rangle}f(\sin\theta-I)e^{-\frac{2\pi im}{\langle b\rangle}I}\diff I\Big)\cos^2\theta\diff\theta\\
&=\int_0^{2\pi}\Big(\int_0^{\langle b\rangle}f(J)e^{-\frac{2\pi im}{\langle b\rangle}(\sin\theta-J)}\diff J\Big)\cos^2\theta\diff\theta\\
&=\int_0^{2\pi}\Big(\int_0^{\langle b\rangle}f(J)e^{-\frac{2\pi i(-m)}{\langle b\rangle}J}\diff J\Big)e^{-\frac{2\pi im}{\langle b\rangle}\sin\theta}\cos^2\theta\diff\theta\\
&=\hat f(-m)\int_0^{2\pi}e^{-\frac{2\pi im}{\langle b\rangle}\sin\theta}\cos^2\theta\diff\theta.
\end{align*}
The desired formula follows using $e^{iz}=\cos z+i\sin z$ and $\sin (\xi\sin(-\theta))\cos^2(-\theta)=-\sin(\xi\sin \theta)\cos^2\theta$.
\end{proof}

We are now in position to prove Theorem \ref{thm:main}.

\begin{proof}[Proof of Theorem \ref{thm:main}]
We assume first that $\langle b\rangle$ does not belong to the set $\mathcal B$ defined in \eqref{setB}. Then, 
$$J_1(2\pi m/\langle b\rangle)\neq0,\quad \forall m\neq0.$$ 
By Lemma \ref{l:fourier}, it follows that $\hat f(-m)=0$ for all $m\neq0$, which implies that $f$, hence $b$, is a constant.  

We assume now that $\langle b\rangle\in \mathcal B$. Then, there exists $m_0\in \N$ such that $J_1(2\pi m_0/\langle b\rangle)=0$. Now, for all $s\in\R$ such that the function
\[
f_s(u)=\frac{1}{\langle b\rangle}+s\cos\Big(\frac{2\pi m_0}{\langle b\rangle}(u-u_0)\Big)
\]
is positive, there exists an orientation preserving diffeomorphism
\begin{equation}\label{e:Fs}
F_s:\R/\langle b\rangle\Z\to\R/\Z,\qquad F_s(u)=\frac{u}{\langle b\rangle}+\frac{s\langle b\rangle}{2\pi m_0}\sin\Big(\frac{2\pi m_0}{\langle b\rangle}(u-u_0)\Big)
\end{equation}
such that $\tfrac{\diff F_s}{\diff u}=f_s$, and the function
\begin{equation}\label{e:bs}
b_s=\frac{\diff F_s^{-1}}{\diff x}
\end{equation}
makes the pair $(g_\con,b_s)$ into a Zoll system. Indeed, $\hat f_s(m)=0$ if $m\neq0,m_0$ and $J_1(2\pi m_0/\langle b\rangle)=0$ ensure that $\hat S(m)=0$ for all $m\neq0$. It is now immediate to check that $f_s$ is positive if and only if $|s|<\langle b\rangle^{-1}$ and that $\inf f_s$ tends to $0$ as $|s|\to \langle b\rangle^{-1}$. This last property implies that $\max b_s$ diverges as $|s|\to \langle b\rangle^{-1}$. Moreover,
\[
\frac{\diff b_s}{\diff s}(0)=-\frac{\diff f_s}{\diff s}(0)=-\cos\Big(\frac{2\pi m_0}{\langle b\rangle}(u-u_0)\Big)\neq0.\qedhere
\]
\end{proof}
\begin{rem}
The proof of Theorem \ref{thm:main} shows that the solution $f_s$ constructed above is the only one giving rise to a Zoll pair if $d(\langle b\rangle)=1$, namely if $\langle b\rangle$ has the property that there exists a unique $m_0$ such that $2\pi m_0/\langle b\rangle$ is a zero of $J_1$ (see the discussion after Figure \ref{figure1}). Moreover, shifting the parameter $u_0$ corresponds to shifting the $x$ variable in $b$ and so the associated $(g_\con,b)$-geodesics get also shifted in the $x$-direction.
\end{rem}
\begin{rem}\label{r:final}
Using the formula for $F_s$ with $u_0=0$ obtained above, we get an explicit expression for $(g_\con,b)$-geodesics $(x_I,y_I):\R\to\T^2$ parametrized by $\theta$. Indeed,
\[
x(I,\theta)=B^{-1}(\sin\theta-I)=F_s(\sin\theta -I),\qquad \frac{\partial x}{\partial I}=-f_s(\sin\theta-I).
\]
Thus, we get
\[
\left\{\begin{aligned}
x_I(\theta)&=\frac{\sin\theta-I}{\langle b\rangle}+\frac{s\langle b\rangle}{2\pi m_0}\sin\Big(\frac{2\pi m_0}{\langle b\rangle}(\sin\theta-I)\Big),\\
y_I(\theta)&=-\frac{1}{\langle b\rangle}\cos\theta+s\int_{-\pi/2}^\theta\sin\theta\cos\Big(\frac{2\pi m_0}{\langle b\rangle}(\sin\theta-I)\Big)\diff\theta,
\end{aligned}\right.
\] 
where we have used \eqref{e:ytheta} to express $y_I$. We used these formulas to produce Figure \ref{figure1} and \ref{figure2} using Python.
\end{rem}

\subsection*{Acknowledgements}
We thank Alberto Abbondandolo and Serge Tabachnikov for fruitful discussions, and Mattia Carlo Sormani for the precious help with the numerical integration.
Luca Asselle is partially supported by the DFG-grant AS 546/1-1 "Morse theoretical methods in Hamiltonian dynamics". Gabriele Benedetti was partially supported by the National Science Foundation under Grant No.~DMS-1440140 while in residence at the Mathematical Sciences Research Institute in Berkeley, California, during the Fall 2018 semester for the program "Hamiltonian systems, from topology to applications through analysis". 

\bibliography{_biblio}
\bibliographystyle{amsalpha}

\end{document}